\documentclass[12pt, oneside]{amsart}       

\usepackage[T1]{fontenc}
\usepackage{mathpazo}
\usepackage{tabularx}
  
\usepackage{hyperref}
\usepackage{booktabs}
\usepackage{xcolor}
\usepackage{tcolorbox}
\hypersetup{
    colorlinks,
    linkcolor={red!50!black},
    citecolor={blue!50!black},
    urlcolor={blue!80!black}
}

\usepackage[all,cmtip]{xy}
\usepackage{graphicx}    
\usepackage{tikz-cd}            
\usepackage{amsthm,amssymb, amsfonts}   
\usepackage{amsaddr} 
\usepackage{mathrsfs}
\newtheorem{theorem}{Theorem}[section]
\newtheorem{proposition}[theorem]{Proposition}
\newtheorem{corollary}[theorem]{Corollary}
\newtheorem{definition}[theorem]{Definition}
\newtheorem{lemma}[theorem]{Lemma}
\theoremstyle{remark}
\newtheorem*{remark}{Remark}
\usepackage{tcolorbox}

\renewcommand{\ge}{\geqslant}
\renewcommand{\geq}{\geqslant}

\renewcommand{\leq}{\leqslant}
 
\usepackage[utf8]{inputenc}

\newcommand{\OO}{\mathcal{O}}

\newcommand{\aff}{\mathbf{A}}

\newcommand{\clm}{\operatorname{colim}}

\newcommand{\dl}{^{\vee}}

\newcommand{\g}{\mathfrak{g}}
\title[Symbols at infinity]{Symbols at infinity and Lie algebras of vector fields}
\author{Emile Bouaziz}
\thanks{Shanghai Institute for Mathematics and Interdisciplinary Sciences, Fudan University. Block A, International Innovation Plaza, No. 657 Songhu Road, Yangpu District, Shanghai, China. Email: ebouaziz@simis.cn}

\begin{document}\maketitle

\begin{abstract} If $X$ is smooth projective complex variety, $D$ is a simple normal crossings divisor and $A$ an ample divisor with support $D$, then there is an associated filtration of the Lie algebra of on the affine variety $U=X\setminus D$ induced from the subspaces $T_X(\log D)(nA)\subset T_U.$ We show that the associated graded Lie algebra is finitely generated in a strong sense. We deduce as a consequence that vector fields on any smooth affine variety form a finitely generated Lie algebra.
\end{abstract}

\section{Introduction} \subsection{Our main results} If $U$ is an affine variety over $\mathbf{C}$, we will consider $\mathrm{Vect}(U)$, the Lie algebra of vector fields on $U$ equipped with the commutator bracket. For $U=\aff^1$ this gives the familar Witt algebra, with generators $L_n=-z^{n+1}\partial_z$ and relations $$[L_n,L_m]=(n-m)L_{n+m}.$$ A quick check confirms that this is generated as a Lie algebra by $L_{-1},L_1,L_2$, and in particular is finitely generated. One of our main theorems, \ref{main}, generalises this to all smooth affine $U$: \begin{theorem}\label{main (intro)} Let $U$ be a smooth affine variety over $\mathbf C$. Then the Lie algebra of vector fields, $\mathrm{Vect}(U)=H^0(U,T_U)$, is finitely generated.
\end{theorem}
The main idea, sketched more thoroughly in \ref{sketch}, consists in choosing a nice enough compactification and filtering $\mathrm{Vect}(U)$ by vanishing order at infinity (in a rational basis of vector fields logarthmic along the boundary). The associated graded is controlled by an algebra of symbols living on the boundary and it is really this algebra that is the focus of much of our efforts.  

In fact, our methods really prove a theorem which is, at least modulo some relatively simple arguments, quite a bit formally stronger than \ref{main}. We state it here:  \begin{theorem}\label{real main - intro} Let $X$ be smooth projective and let $D$ an SNC divisor, and $A$ an ample effective divisor with support $D$. Consider the filtered Lie algebra $$\mathfrak{g}=H^0(X,T_X(\log D)(\infty A)),\,\,F^n\g:=H^0(X,T_X(\log D)(nA)),$$  Then for any $N\geq 0$ the graded Lie algebra $\mathrm{Gr}_{\geq N}^F\g$ is finitely generated. \end{theorem}

\begin{remark} Just so the inputs are clear here, we could take $X$ an elliptic curve, $D=\infty$ a base point and $A=3\infty$. Or, taking a point $p$ in a line $\ell\subset \mathbf{P}^2$, we could take $X=\mathrm{Bl}_p\mathbf{P}^2$ with $D=\tilde{\ell}+E$ and $A=2\tilde{\ell}+E,$ where $\tilde{\ell}$ is the strict transform of $\ell$ and $E$ the exceptional divisor.\end{remark}

\begin{remark} The reader will notice above the appearance, perhaps surprising, of the subscript $_{\geq N}$ in $\mathrm{Gr}^F_{\geq N}\g$. Of course, this is a priori stronger than $\mathrm{Gr}^F\g$ being finitely generated. The necessity of this stronger condition in order to prove Theorem \ref{main (intro)} is, roughly, a result of only knowing exactness of $H^0$ after twisting by a suitably high power of an ample line bundle. This is also (one of) the reasons we require $A$ to be ample.  \end{remark}

\subsubsection{Some comments on prior work} In previous joint work of the author, \cite{BB}, the main result of this note appeared as a conjecture. In the case of curves, Theorem 15 in \cite{Mathieu} proves the considerably stronger result that any Lie subalgebra is finitely generated.\footnote{This result certainly does not generalize, $\mathbf{C}[x]\partial_y\subset\mathrm{Vect}(U)$ is not finitely generated.} In the recent preprint \cite{AR} the authors ask whether vector fields on the Koras-Russel cubic form a finitely generated Lie algebra. Our work in particular answers this question.

\subsubsection{The appendix} We include an appendix. It contains a proof, found by ChatGPT 6, of the weaker of our finite generation results, theorem \ref{main (intro)}. GPT was prompted by the author after the writing of the main text. Our prompt gave no hints and an answer was obtained in about ten minutes, with human checking taking not much longer. Indeed, it is a quite remarkably simple and clever proof, maybe even a little too clever.   It would seem, although we cannot be sure and the author's biases are here quite plain, that some of the identities relevant for ChatGPT's proof may be inspired by the arguments of the papers \cite{BF18}, \cite{BF16} and \cite{BR23}.

 \textbf{The reader who is primarily interested in the purely Lie theoretic question of finite generation of vector fields would probably be best served by only reading the appendix. The geometrically inclined reader will perhaps find more meaning in the main text.}

\subsection{AI usage and Acknowledgements} The author made limited use of AI for the main body of this text - Gemini supplied the proof, presumably mundane for more adept geometers, of lemma \ref{blowups} and ChatGPT 5.6 located the references to Stacks Project. On the other hand the appendix, as explained above, \emph{is wholly the work of ChatGPT 6}, performed after the writing of the main text. 

On a more human note, the author has benefited over the years from numerous conversations about Lie algebras of vector fields with Yuly Billig, Colin Ingalls and Henrique Rocha. In particular the relevance of algebras of sections $H^0(Y,L^{\otimes\bullet})$, and the Euler Lie algebras they produce, became clear during a conversation about $\mathcal{AV}$-modules with the above named mathematicians in 2023.

\subsection{Our objects of study and a sketch of the methods} \subsubsection{Filtrations from nice compactifications}\label{sketch}

The proof chooses an embedding $o:U\to X\footnote{$i$ and $j$ are going to be used as indices so we avoid the usual $(i,j)=(\mathrm{closed},\mathrm{open})$ notation in favour of $(\iota,o)$.}$, with $X$ smooth projective with SNC boundary $D=\sum D_i$ supporting an effective ample divisor $A=\sum a_iD_i$. \begin{remark}At a first pass the reader might assume for simplicity that $D$ is already ample, and $A=D$ so that all $a_i=1$. This special case does not suffice to prove the main theorem, but certainly gives a very good sense of the arguments.\end{remark} We let $\iota_A$ denote the inclusion of $A$ and then consider the filtration $F$ of the sheaf $o_*T_U$ given by $F^no_*T_U:=T_X(\log D)(nA)$ for $n\geq0$. This is an exhaustive Lie algebra filtration. Its associated graded agrees in positive degrees with a sheaf of symbols, $$\mathcal{S}(A):=\iota_{A,*}\iota_A^*(T_X(\log D)\otimes\OO_X(A)^{\otimes\bullet}).$$  At the level of global sections this induces a filtration on $\mathrm{Vect}(U)$ given by $$F^n\mathrm{Vect}(U):=H^0(X,T_X(\log D)(nA)).$$ We then show the stronger property that $\mathrm{Gr}^F\mathrm{Vect}(U)$ is finitely generated, and indeed that $\mathrm{Gr}_{\geq N}^F\mathrm{Vect}(U)$ is finitely generated for all $N\geq0$. By the identification of the associated graded above, coupled with Serre vanishing (recall $A$ is ample), it suffices to show that the graded Lie algebra $H^0(X,\mathcal{S}(A))_{\geq N}$ is finitely generated for any $N$. This is the main focus of our efforts.

Finite generation is proven by an induction on the number of components of $D$. If $B\leq A$ is any subdivisor we can also consider $$\mathcal{S}(B):=\iota_{B,*}\iota_B^*(T_X(\log D)\otimes\OO_X(A)^{\otimes\bullet}),$$ then there is a natural restriction map of Lie algebras $$\mathcal{S}(A)\to\iota_{B,*}\iota_B^*\mathcal{S}(A)\simeq\mathcal{S}(B).$$ The simplest case $B=D_i$ has some additional features that make it more managable. We set $L=\OO_X(A)$ and $\Delta_i=\sum_{j\neq i}D_{ij}$ in what follows. Then we have an exact sequence
$$0\to\OO_{D_i}\to\iota_i^*T_X(\log D)\to T_{D_i}(\log\Delta_i)\to0.$$ In local coordinates with respect to which $D_i$ is cut out by $z_i=0$, the image of $1\in\OO_{D_i}$ is given by $\iota_i^* (z_i\partial_i)\footnote{The bracketing here is meant to emphasize that we consider $z_i\partial_i$ as a section of the log tangent sheaf. Considered as a section of the tangent sheaf, the resulting restriction is of course $0$.}$ We denote this section by $e$. Logarithmic vector fields act on $L$ through the inclusion $L=\OO_X(A)\subset o_*\OO_U$, and this action descends to the restrictions and their tensor powers. Then the short exact sequence above supplies a filtration of $\mathcal{S}(D_i)$, with subspace denoted $\mathcal{N}(D_i)$ and quotient $\mathcal{T}(D_i)$. Denoting by $\mathfrak{s}(D_i)$ and $\mathfrak{n}(D_i)$ the resulting Lie algebras of global sections, perhaps the main point is that $\mathfrak{n}(D_i)$ can be identified purely algebraically in terms of the graded ring $H^0(D_i,\iota_i^*L^{\otimes\bullet})$ as what we call an \emph{Euler Lie algebra}, described in \ref{euler intro} below.

Further, there is a natural class of modules for this subalgebra, depending on parameters $\lambda,\mu\in\mathbf C$ and a coherent sheaf $F$ on $D_i$, which we will denote $\mathfrak{m}_{\lambda,\mu}(F)$ and describe properly later, suffice it to say for now that the the module obtained by taking global sections of $\mathcal{T}(D_i)$ is also of this form. Then, assuming that $L$ is ample, we prove:

\begin{lemma}\label{symb finite (intro)} For all $N\geq0$, the Lie algebra $\mathfrak{s}(D_i)_{\geq N}$ is finitely generated. \end{lemma}
\begin{lemma}\label{eul mod fin gen (intro)} For all $N\geq0$, the subalgebra $\mathfrak{n}(D_i)_{\geq N}$ is finitely generated and $\mathfrak{m}_{\lambda,\mu}(F)_{\geq N}$ is a finitely generated module for it.\end{lemma}

\subsubsection{Lie algebras coming from graded algebras}\label{euler intro} Lemmas \ref{symb finite (intro)} and \ref{eul mod fin gen (intro)} recorded above are consequences, modulo some simple d\'{e}vissage and standard results on coherent cohomology, of some results of purely algebraic nature. These results are collected in section \ref{just algebra}. They are, whilst quite simple, maybe of interest and we choose to record them here as well. More thorough definitions can be found further in the text. If $R$ is a commutative graded algebra and $\kappa\in\mathbf C^*$ is fixed then we denote by $R_{\mathrm{Eul}}^\kappa$ the Lie algebra with bracket $$[r,s]=\kappa\big(\deg(r)-\deg(s)\big)rs.$$  If $M$ is a graded module for $R$ and $\lambda,\mu\in\mathbf C$ are two parameters then we denote by $M^\kappa_{\lambda,\mu}$ the module for $R_{\mathrm{Eul}}^\kappa$ with action written $r\otimes m\mapsto r\cdot m$ and defined as$$r\cdot m:=\big(\lambda+\mu\deg(r)-\kappa\deg(m)\big)rm.$$ Then we will prove the following lemmas, both purely algebraic and both rather elementary:\begin{lemma} Let $R=\oplus_{n\geq 0} R_n$ with $R_0$ finite dimensional and $R$ finitely generated over $R_0$. Then for any $N\ge 0$ the Lie algebra $R_{\mathrm{Eul},\geq N}$ is finitely generated. \end{lemma} \begin{lemma} If $M$ is a finitely generated graded $R$-module then for any $N\geq0$, and $\lambda,\mu\in\mathbf C$, the module $M^\kappa_{\lambda,\mu,\geq N}$ is finitely generated for $R_{\mathrm{Eul},\geq N}$. \end{lemma}

\section{Basics and notation}This section is dedicated to some simple background recollections, much of which was already sketched more hastily in the introduction. In the first subsection we set up some terminology regarding graded Lie algebras that will make use of throughout. We also record some very simple d\'{e}vissage arguments which we will make use of later. 

The second subsection is much more geometric. In it we fix some notation for logarithmic vector fields and their restrictions to the boundary. We also introduce some of the sheaves of Lie algebras which will be important to us throughout the remainder of this note.

\subsection{Graded Lie algebras}\label{Lie bit} Let $\g_\bullet=\oplus_n\g_n$ be a graded Lie algebra. We write $$\g_{\geq n}=\oplus_{i\geq n}\g_i.$$ \begin{definition}A graded Lie algebra $\g_\bullet$ is called persistently finitely generated if, for all $n\geq 0$, the Lie algebra $\g_{\geq n}$ is finitely generated. A graded module $M_\bullet$ for $\g_\bullet$ is called persistently finitely generated if, for all $n\geq 0$, $M_{\geq n}$ is finitely generated as a module for $\g_{\geq n}$. \end{definition}\begin{remark}Persistent finite generation of course implies finite generation. If each $\g_n$ is finite dimensional then finite generation of any $\g_{\geq N}$ implies finite generation of $\g_{\geq M}$ for all $M\leq N$. Finite generation need not imply persistent fintie generation, indeed this fails for $\g$ the free Lie algebra on two generators with grading induced by commutator length. \end{remark} We record two trivial lemmas for organisational purposes. \begin{lemma}\label{gr finite} Let $\g=\clm_{n\geq 0} F_n\g$ be a filtered Lie algebra. If the associated graded $\mathrm{Gr}^F\g$ is finitely generated then so too is $\g$. \end{lemma} \begin{proof}Pick homogeneous generators and lift them to $\g$ via the symbol map (of vector spaces). An evident downward induction implies that these generate. \end{proof}

\begin{lemma} \label{gr pers fin}Let $\g=\clm_{n\geq 0} F_n\g$ be a filtered Lie algebra with all $F_n\g$ finite dimensional. Let $\mathfrak{s}_\bullet$ be a graded Lie algebra which is persistently finitely generated. If there is a morphism $\mathrm{Gr}^F_\bullet\g\to\mathfrak{s}_\bullet$ which is an equivalence on all sufficiently high truncations $_{\geq N}$, then $\g$ is finitely generated.   \end{lemma} \begin{proof} For $n$ sufficiently large $\mathrm{Gr}^F_{\geq n}\g$ is finitely generated as it is isomorphic to $\mathfrak{s}_{\geq n}$ and $\mathfrak{s}$ is persistently finitely generated. Since all $F_i\g$ are finite dimensional, adjoining bases of the low degree terms gives a finite generating set for $\mathrm{Gr}^F\g$. The previous lemma then implies that $\g$ is finitely generated.\end{proof}

\begin{definition} Let $\g_\bullet$ be a graded Lie algebra and let $\mathfrak{m}_\bullet$ be a graded module for $\g_\bullet$. We say that $\mathfrak{m}_\bullet$ is \emph{persistently finitely generated} if, for every $n\geq0$, the tail $\mathfrak{m}_{\geq n}$ is finitely generated as a module for $\g_{\geq n}$.
\end{definition}

Finally, we record a slightly dry and technical lemma. We choose to do so in order that the logic of the proof of lemma \ref{q fin gen} is made clearer. The reader should skip what follows until they encounter lemma \ref{q fin gen}. 

\begin{lemma}\label{tech} Suppose we have a Lie algebra $\g$ with ideal $\mathfrak{K}$ and quotient $\mathfrak{h}$. Assume further that $\mathfrak{K}$ is equipped with a finite length decreasing filtration by $\g$-submodules $$\mathfrak{K}=\mathfrak{K}_0\supset \mathfrak{K}_1\supset\cdots,$$ and assume that the $\g$ action on the associated graded pieces $\mathfrak{K}_c/\mathfrak{K}_{c+1}$ factors through some finitely generated quotient $\g\to\mathfrak{s}.$ Then if $\mathfrak{h}$ is finitely generated and each associated graded piece $\mathfrak{K}_c/\mathfrak{K}_{c+1}$ is finite type as a module for $\mathfrak{s}$, then $\g$ is finitely generated.  \end{lemma} \begin{proof} Choose finitely many generators of $\mathfrak{s}$, lift them to $\g$ and let $\mathfrak{l}$ be the Lie algebra they generate. Then each $\mathfrak{K}_c$ is a module for $\mathfrak{l}$ and the associated gradeds are finitely generated $\mathfrak{l}$-modules since they are pulled back from $\mathfrak{s}$, and are finitely generated $\mathfrak{s}$ modules. Then by d\'{e}vissage $\mathfrak{K}$ is a finitely generated $\mathfrak{l}$-module. We pick finitely many generators for $\mathfrak{l}$, lifts of finitely many generators of $\mathfrak{h}$, and finitely many generators of the $\mathfrak{l}$-module $\mathfrak{K}$. The union of these is then easily seen to generate. \end{proof}

\subsection{Geometric preparations}Let $D=\sum_i D_i$ be a simple normal crossings divisor in a smooth projective variety $X$. Write $\iota:D\to X$ for the inclusion of $D$, $o$ for the inclusion of the open complement $U$, and $\iota_j$ for the inclusion of the $j$th component and $I_j$ for the corresponding ideal sheaf. If $B=\sum_j b_j D_j$ is an effective divisor with support contained in a union of smooth components of $D$, we denote by $\iota_B:B\to X$ the inclusion and $\iota_B$ the corresponding ideal sheaf. For fixed $j$ we let $\Delta_j:=\sum_{i\neq j}D_{ij}$. Let $T_X(\log D)$\footnote{We have chosen this notation over the more customary $T_X(-\log D)$ as it is less ugly and there is no real risk of confusion.} denote the sheaf of logarithmic vector fields, ie vector fields preserving the ideal sheaves $I_j$ of each $D_j$. It is closed under the bracket of vector fields and there are exact sequences
$$0\to T_X(\log D)\to T_X\to\bigoplus_j\iota_{j,*}N_{D_j/X}\to0,$$
$$0\to\OO_{D_j}\to\iota_j^*T_X(\log D)\to T_{D_j}(\log\Delta_j)\to0.$$
 Both exact sequences encode residues. Indeed the first is, after a rotation in the derived category, the coherent dual to the usual residue sequence. The inclusion of the kernel in the second is dual to the map $\iota_j^*\Omega_X(\log D)\to\OO_{D_j}$, adjoint to the usual residue map. More canonically the kernel is $\mathrm{End}_{D_j}(N^*_{D_j/X})$, with the identity corresponding locally to $z_j\partial_j$. It is the second sequence of which we make considerable use throughout. We now fix another effective divisor $$A=\sum a_jD_j, \,(a_j\geq 0),$$ whose support is a union of components of $D$. Eventually we will assume that $A$ is ample, hence the name. 
 
\subsubsection{Actions of log vector fields} The sheaf $T_X$ acts via derivations on the sheaf of rational functions $o_*\OO_U$, preserving the subspace $\OO_X$. Its logarithmic version $T_X(\log D)$ further preserves the subsheaves $\OO_X(B)\subset o_*\OO_U$, whenever $B=\sum b_j D_j$ is some divisor whose support is union of components of $D$. Recalling that logarithmic vector fields, by definition, preserve each ideal sheaf $I_j$, and hence each $I_B$, it follows that $\iota_B^*T_X(\log D)$ inherits a natural Lie bracket. Note that, whilst clear, this is not purely formal, indeed $\iota_B^*T_X$ does not inherit such a bracket in a natural fashion.

\begin{remark}\label{coord} Analytic locally we may choose coordinates $(z_1,\cdots,z_n)$ on $X$ so that $D$ and $A$ are given, for some $r\leq n$, by $$D=\{z_1\cdots z_r=0\}\,\,A=\{z_1^{a_1}\cdots z_r^{a_r}=0\}.$$ We will sometimes write $z^B$ for the monomial cutting out the effective divisor $B=\sum_j b_jD_j$, which is to say we set $$z^B=\prod_i z_i^{b_i}.$$ Then $T_X$ is trivialised by the basis $\{\partial_1,\cdots,\partial_n\}$ and $T_X(\log D)$ by $$\{z_1\partial_1,\cdots,z_r\partial_r,\partial_{r+1},\cdots,\partial_n\}.$$ The line bundle $L=\OO_X(A)\subset o_*\OO_{X\setminus D}$ is given by $z^{-A}\OO_X\subset o_*\OO_{X\setminus D}$ and has trivialising section $z^{-A}.$ \end{remark}

As we have seen above, $\iota_B^*T_X(\log D)$ inherits a natural Lie bracket, and $T_X(\log D)$ acts on $L=\OO_X(A)$. Further, it is easy to check that there is an induced action of $\iota_B^*T_X(\log D)$ on $\iota_B^*L$, whence on all of its tensor powers by the Leibniz rule.

  \begin{definition} \label{Q(B) bracket def} We denote by $\mathcal{S}(B)_\bullet=\iota_{B,*}\iota_B^*\big(T_X(\log D)\otimes L^{\otimes\bullet}\big)$ the sheaf of graded Lie algebras equipped with the bracket given  $$[vs^{ n},wt^{ m}]=[v,w]s^{n}t^{m}-nvw(s)s^{n-1}t^{ m}+mwv(t)s^{ n}t^{m-1},$$ where $v,w$ are local sections of $T_X(\log D)$ and $s,t$ local sections of $\OO_X(A)$. We denote by $\mathfrak{s}(B)$ the Lie algebra of global sections. \end{definition}

\begin{remark} The presence of $\iota_{B,*}$ is really just some good manners, grammatically speaking. It is of no real substance and ensures only that the various sheaves $\mathcal{S}(B)$ live in the same category. Further, these good manners are discarded in \ref{S algebra}, for which we apologise in advance.  \end{remark}

We have the following elementary lemma.
\begin{lemma}\label{symbols} The subsheaves $F^no_*T_U:=T_X(\log D)(nA)$, for $n\geq0$, define an exhaustive Lie algebra filtration. There is a natural surjective morphism of graded sheaves of Lie algebras
$$\mathrm{Gr}_\bullet^F o_*T_U\longrightarrow\mathcal{S}(A),$$
which is an equivalence in every positive degree. 
\end{lemma} \begin{proof} There is a certainly a map at the level of graded sheaves, just coming from tensoring the short exact sequences $$0\to \OO_X((n-1)A)\to\OO_X(nA)\to\iota_{A,*}\iota_A^*\OO_X(nA)\to 0$$ by $T_X(\log D)$. Now the claim can be checked locally. Take local sections $s^{n}v,t^mw$ of the left hand side, with $v$ and $w$ logarithmic along $D$ and $s,t$ local sections of $\OO_X(A)$. Then the bracket is computed by the Leibniz formula and seen to agree, upon restriction to $A$, with the bracket defined on $\mathcal{S}(A)$.  \end{proof}

\begin{remark} Note that the fact that this is only true in positive degrees is a consequence of defining $F^{-1}=0$. It is of no consequence for our results or arguments and can just be ignored. \end{remark}
\subsubsection{Restriction to a component} The simplest choice of $B$ is given by $B=D_i$ for some $i$. This case is somewhat special, and we summarise this specialness now. Recall we have fixed $A=\sum a_iD_i$ and we write $L=\OO_X(A)$. The natural morphism $\OO_{D_i}\to \iota_i^*T_X(\log D)$ supplies a section $e$, which we take to be locally given by the restriction of $z_i\partial_i$. It satisfies the following properties, all readily confirmed: \begin{itemize}\item The image of $e$ in $T_{D_i}(\log \Delta_i)$ vanishes. \item $e$ is central, ie it commutes with any local section of $\iota_i^*T_X(\log D)$. \item It acts on $\iota_i^*L$ by multiplication by $-a_i$.\end{itemize}

\subsubsection{The symbol algebra on a component}\label{S algebra} We focus now on $\mathcal{S}(D_i)$. We will allow ourselves some mild abuse of notation, in that we will drop the pushforward $\iota_{i,*}$ in the definition, and instead work on $D_i$ itself. Whilst strictly ungrammatical in that $\mathcal{S}(D_i)$ has now acquired two distinct meanings, we hope this does not cause any confusion.

So, with this notational abuse in mind, we remind the reader that we have set $\mathcal{S}(D_i)=\iota_i^*T_X(\log D)\otimes(\iota_i^*L)^{\otimes\bullet}$, considered as a sheaf of graded Lie algebras with bracket $$[a\otimes s^n,b\otimes t^m]=[a,b]\otimes s^nt^m-na\otimes t^ms^{n-1}b(s)+mb\otimes s^nt^{m-1}a(t),$$
where $a,b$ are local sections of $\iota_i^*T_X(\log D)$ and $s,t$ are local sections of $\iota_i^*L$. Recalling that we write $e$ for the section of $\iota_i^*T_X(\log D)$ supplied by the dual residue sequence, we will write $$\mathcal{N}(D_i):=(\iota_i^*L)^{\otimes\bullet}e\subset \mathcal{S}(D_i),$$ and set $\mathcal{T}(D_i)=\mathcal{S}(D_i)/\mathcal{N}(D_i)$, considered as a module for $\mathcal{N}(D_i)$.

\begin{remark}\label{e wts} The section $e$ of $\mathcal{N}(D_i)$ acts on the weight $n$ piece of $\mathcal{S}(D_i)$ as multiplication by $-a_in$. Indeed this is seen from the above formula, in which the first weight $m$ is zero and the first bracket vanishes as $e$ is central in $\iota_i^*T_{D_i}(\log\Delta_i).$ \end{remark}

\begin{definition} We denote by $\mathfrak{s}(D_i)$ and $\mathfrak{n}(D_i)$ the Lie algebras obtained by applying $H^0(D_i,-)$ to $\mathcal{S}$ and $\mathcal{N}$, and by $\mathfrak{t}(D_i)$ the module for $\mathfrak{n}(D_i)$ obtained by applying $H^0(D_i,-)$ to $\mathcal{T}$. \end{definition} 

\subsubsection{Modules for the normal algebra}\label{norm alg brack} $\mathcal{N}(D_i)$ is a very easy Lie algebra to understand, indeed for local sections $u$ of $(\iota_i^*L)^{\otimes p}$ and $v$ of $(\iota_i^*L)^{\otimes q}$, and recalling $e=z_i\partial_i$, a direct computation gives $$[ue,ve]=a_i(p-q)uv\,e.$$ So the Lie algebra $\mathcal{N}(D_i)$ is obtained from the structure of a graded commutative algebra on its underyling graded sheaf in a very simple way. In section \ref{just algebra} formulate this more precisely and devote a little time to the study of such Lie algebras. 
\\
\begin{definition} Let $F$ be a coherent sheaf on $D_i$ and let $\lambda,\mu\in\mathbf C$. Set $$\mathcal{M}_{\lambda,\mu}(D_i,F)=F\otimes(\iota_i^*L)^{\otimes\bullet},$$ which we will consider as a graded module for $\mathcal{N}(D_i)$ by declaring that $$(ue)\cdot w=(\lambda+\mu p-a_i n)uw$$ whenever $u$ is a local section of $(\iota_i^*L)^{\otimes p}$ and $w$ is a local section of $F\otimes(\iota_i^*L)^{\otimes n}$. When context is clear we abbreviate this to $\mathcal{M}_{\lambda,\mu}(F)$.
\end{definition}

\begin{definition} $\mathfrak{m}_{\lambda,\mu}(D_i,F):=H^0(D_i,\mathcal{M}_{\lambda,\mu}(D_i,F)),$ considered as a graded module for $\mathfrak{n}(D_i)$. We sometimes abbreviate this as $\mathfrak{m}_{\lambda,\mu}(F)$. \end{definition} \begin{lemma}\label{tgt term} There is an equivalence of $\mathcal{N}(D_i)$-modules $$\mathcal{T}(D_i)\simeq\mathcal{M}_{0,0}(T_{D_i}(\log\Delta_i)).$$ \end{lemma}\begin{proof} Tensoring $\iota_i^*L^{\otimes\bullet}$ by the exact sequence $$0\to\OO_{D_i}\to\iota_i^*T_X(\log D)\to T_{D_i}(\log\Delta_i)\to 0$$ supplies the claim at the level of sheaves. Now choose (locally) a lift of $w\in\mathcal{T}(D_i)$ to $\widetilde{w}\in\mathcal{S}(D_i)$, and let $u$ be a local section of $\iota_i^*L^{\otimes\bullet}$. The action is obtained from $[ue,\widetilde{w}]=u[e,\widetilde{w}]-\widetilde{w}(u)e$ by projecting onto the quotient. The second term vanishes in the quotient and the first is just $-a_inuw$ by remark \ref{e wts}. This matches the definition of $\mathcal{M}_{0,0}(T_{D_i}(\log\Delta_i))$, and we are done. \end{proof}

\section{Some graded algebra}\label{just algebra}

This section provides a slightly broader context for studying the normal algebras and their modules. The results we prove here are of a purely algebraic nature but, of course, our actual intention is to apply them to the cases considered in the previous section. Nonetheless we prefer to separate out this section so that our arguments and objects of study appear in their natural home.

\subsection{The Euler Lie algebra} Let $R=\bigoplus_{n\geq0}R_n$ be a graded commutative $\mathbf C$-algebra and fix $\kappa\in\mathbf C^*$. We denote by $R_{\mathrm{Eul}}^\kappa$ the graded Lie algebra with underlying graded vector space $R$ and bracket $$[a,b]=\kappa(p-q)ab,\,\, (a\in R_p,b\in R_q).$$ \begin{remark} Obviously $\kappa$ is not of any real importance here, and rescaling supplies an isomorphism with the corresponding Lie algebra at $\kappa=1$. The inclusion here is only to agree with the bracket from \ref{norm alg brack}, which had a coefficient of $a_i$. \end{remark}

\begin{remark} This is a specific instance of a more general construction. Let $S$ be a commutative algebra with derivation $\delta$. Then $S\delta\subset\mathrm{Der}(S)$ is a Lie subalgebra, with $[a\delta,b\delta]=(a\delta(b)-b\delta(a))\delta$. If we take $S=R$ and consider the derivation $\delta=-\kappa\nu$, with $\nu$ the Euler vector field supplied by the grading, we obtain the above example. \end{remark}

\begin{lemma}\label{euler algebra finite generation} Let $R=\bigoplus_{n\geq0}R_n$ be a finitely generated graded commutative $\mathbf C$-algebra and suppose that $R_0$ is finite dimensional. Then $R_{\mathrm{Eul}}^\kappa$ is persistently finitely generated.
\end{lemma} \begin{proof} Fix $N\geq0$. Let $x_i$ be homogeneous generators of $R$ over $R_0$ of strictly positive degrees $d_i$, and choose $p\geq N$ so that $d_i$ divides $p$ for all $i$. If there are no such generators then $R=R_0$ is finite dimensional and the claim is immediate. Let $R(p)=R_0[R_p]$ be the subalgebra generated by $R_p$. Then $R$ is a finite type module for $R(p)$: every $x_i$ is integral over $R(p)$, satisfying $T^{p/d_i}-x_i^{p/d_i}=0$ with $x_i^{p/d_i}\in R_p$. Then $R_{\geq N}$ is also a finite module for $R(p)$ as $R(p)$ is finitely generated whence Noetherian. Let $y_j$ be homogeneous generators for the $R(p)$-module $R_{\geq N}$. We claim that the finite dimensional subspace $R_p+R_{2p}+\sum_jR_0y_j$ generates $R_{\mathrm{Eul},\geq N}$ as a Lie algebra. Indeed, any element can be written as a sum of elements $$cr_1\cdots r_l y_j,\,\, (r_i\in R_p,c\in R_0).$$ We induct on $l$, so we may assume that we have already obtained $cr_2\cdots r_l y_j$. By definition of the bracket,
$$[cr_2\cdots r_l y_j,r_1]=\kappa\big(\deg(y_j)+(l-2)p\big)cr_1\cdots r_l y_j.$$
If the coefficient vanishes then $cr_1\cdots r_l y_j$ has degree $2p$. These elements have been included by hand, and so we are done. \end{proof}

\subsection{Modules for the Euler Lie algebra}

Let $M=\bigoplus_{n\geq0}M_n$ be a graded module for $R$. Fix $\lambda,\mu\in\mathbf C$ and let $R_{\mathrm{Eul}}^\kappa$ act on $M$ by $$a\cdot m=(\lambda+\mu p-\kappa n)am$$ for $a\in R_p$ and $m\in M_n.$ We denote the resulting module by $M_{\lambda,\mu}^\kappa$.

\begin{lemma}\label{eul mod fin gen} Let $R$ satisfy the hypotheses of Lemma \ref{euler algebra finite generation}, and let $M$ be a finitely generated graded $R$-module. Then, for every $\lambda,\mu\in\mathbf C$, $M^\kappa_{\lambda,\mu}$ is persistently finitely generated as a module for $R_{\mathrm{Eul}}^\kappa$.
\end{lemma}

\begin{proof} Let $N$ be given. We must prove that $M^\kappa_{\lambda,\mu,\geq N}$ is finitely generated as a module for $R_{\mathrm{Eul},\geq N}$. The case $R=R_0$ is immediate since $M$ is then finite dimensional. Otherwise choose $p$ as above and note that $M_{\geq N}$ is finite type as an $R(p)$-module, since $M$ is a finite type module for $R$ and $R$ a finite type module for $R(p)$. Choose finitely many homogeneous generators $m_j$ for $M_{\geq N}$ as an $R(p)$-module. We claim that
$$\sum_jR_0m_j+M_{p+(\lambda+\mu p)/\kappa}$$
generates.\footnote{The latter summand is interpreted as $0$ if its index is not an integer greater than or equal to $N$.} Every element is a sum of terms $cr_1\cdots r_lm_j$, with $r_i\in R_p$ and $c\in R_0$. Again we induct on $l$. The action gives
$$r_1\cdot(cr_2\cdots r_lm_j)=\big(\lambda+\mu p-\kappa\deg(cr_2\cdots r_lm_j)\big)cr_1\cdots r_lm_j.$$
We are done unless $cr_2\cdots r_lm_j$ has degree $(\lambda+\mu p)/\kappa$, in which case the desired element has degree $p+(\lambda+\mu p)/\kappa$ and has been included by hand.
\end{proof}

\subsection{Some geometric applications}

Let $Y$ be a projective variety with an ample line bundle $L$ and a coherent sheaf $F$. We set
$$R(Y,L)=\bigoplus_{n\geq0}H^0(Y,L^{\otimes n}),\,\,
M(Y,F,L)=\bigoplus_{n\geq0}H^0(Y,F\otimes L^{\otimes n}).$$
Then $R(Y,L)$ is a finitely generated graded $\mathbf C$-algebra and $M(Y,F,L)$ is a finite type graded module over it \cite[Tag 0B5T, Lemma 30.16.1(1), (5)]{Stacks}. We apply this with $Y=D_i$ and the line bundle $\iota_i^*L$, where $L=\OO_X(A)$ as above.

\begin{proposition}\label{norm fin gen} Suppose that $L$ is ample and let $F$ be a coherent sheaf on $D_i$. Then $\mathfrak{n}(D_i)$ is persistently finitely generated. Moreover, for every $\lambda,\mu\in\mathbf C$, $\mathfrak{m}_{\lambda,\mu}(D_i,F)$ is persistently finitely generated as a module for $\mathfrak{n}(D_i)$.
\end{proposition} \begin{proof} We claim there are identifications: $$R(D_i,\iota_i^*L)_{\mathrm{Eul}}^{a_i}\simeq\mathfrak{n}(D_i),$$ $$M(D_i,F,\iota_i^*L)_{\lambda,\mu}^{a_i}=\mathfrak{m}_{\lambda,\mu}(D_i,F).$$ Indeed, the map $u\mapsto ue$ supplies the first and the second follows from the definition of the actions. The claim then follows from the preceding lemmas.
\end{proof}

\begin{corollary}\label{symb finite} If $L$ is ample, the graded Lie algebra $\mathfrak{s}(D_i)$ is persistently finitely generated. \end{corollary}

\begin{proof} Tensor the normal-tangent sequence with $(\iota_i^*L)^{\otimes\bullet}$. Serre vanishing implies that for every sufficiently large $N$ we have a short exact sequence
$$0\to\mathfrak{n}(D_i)_{\geq N}\to\mathfrak{s}(D_i)_{\geq N}\to\mathfrak{t}(D_i)_{\geq N}\to0.$$
This is seen to be an exact sequence of modules for $\mathfrak{n}(D_i)_{\geq N}$, with quotient term identified with $\big(\mathfrak{m}_{0,0}(T_{D_i}(\log\Delta_i))\big)_{\geq N}$. Thus the normal subalgebra is finitely generated as a Lie algebra and the quotient is finitely generated as a module for it, and we are done (for example by lemma \ref{gr finite} in the case of length two filtration.)\end{proof}

\section{Induction on components of the boundary}

\subsubsection{Some orienting remarks} This section contains the main technical lemma on which we will rely in order to prove our main theorem (in the next section). The main difficulty arises from multiple irreducible components in the boundary of a compactification. In fact, if $U$ admits a smooth projective compactification $X$ with smooth ample boundary $D$, then the results proven above easily imply finite generation of $\mathrm{Vect}(U)$. 
Indeed we can take the filtration $$F^n\mathrm{Vect}(U)=H^0(X,T_X(\log D)(nD)),$$ then we have a morphism $\mathrm{Gr}^F\mathrm{Vect}(U)\to\mathfrak{s}(D)$ which is an equivalence in sufficiently high degrees. By lemma \ref{symb finite} we are done. \begin{remark}Note that admitting a compactification with smooth boundary is a very restrictive condition, with a basic obstruction coming from weights in Hodge theory. Indeed any non zero cohomology group, $\mathrm{Gr}_W^{i+j}H^i$, with $j\geq 2$ is an obstruction.\end{remark}The argument in the case of smooth boundary does not directly generalise. It made substantial use of the section $e$ of $\iota_D^*T_X(\log D)$, and the resulting normal subalgebra. For a boundary with more components, the normal-tangent sequence and its canonical section are available on each smooth component $D_i$, but \emph{do not} give such a sequence on the entire boundary. We are forced then to argue inductively, adding in one (thickened) component at a time.

\subsubsection{The proof of lemma \ref{q fin gen}}We now prove the main lemma of this section. As already mentioned, the elementary Lie theoretic mechanism which forms a part of the proof was recorded in lemma \ref{tech} and the reader should consult it during the reading of the following proof. It may help the reader to have in mind the correspondence between the notation of lemma \ref{tech} and the proof of lemma \ref{q fin gen}, given schematically as $$\mathcal{S}(B)\leftrightarrow\g,\, \mathcal{S}(B')\leftrightarrow\mathfrak{h},\,\mathcal{K}\leftrightarrow\mathfrak{K}.$$ 

\begin{lemma}\label{q fin gen} If $A$ is ample then Lie algebra $\mathfrak{s}(A)$ is persistently finitely generated.
\end{lemma}

Most of the lemma is subsumed in the following filtration argument at the sheaf level, which we prove first as a preliminary lemma. In the following we let $B'< B\leq A$ be effective subdivisors satisfying $B=B'+a_iD_i$, so that in particular $D_i$ does not appear in $B'$.

\begin{lemma}\label{prelim} The kernel $\mathcal{K}_0$ of the restriction map $\mathcal{S}(B)\to\mathcal{S}(B')$ has a finite length filtration by $\mathcal{S}(B)$-submodules $\mathcal{K}_0\supset\mathcal{K}_1\supset\cdots\supset\mathcal{K}_{a_i}=0$ so that the following two properties are satisfied:
\begin{itemize}
\item The action of $\mathcal{S}(B)$ on the associated graded $\mathrm{Gr}\mathcal{K}$ factors through the restriction morphism $\mathcal{S}(B)\to\mathcal{S}(D_i)$.
\item Considered as modules for the normal subalgebra $\mathcal{N}(D_i)$ and setting $F_c=\iota_i^*\OO_X(-B'-cD_i)$, the associated graded pieces sit in short exact sequences
$$0\to\iota_{i,*}\mathcal{M}_{c,a_i}(F_c)\to\mathrm{Gr}_c\mathcal{K}\to\iota_{i,*}\mathcal{M}_{c,0}(F_c\otimes T_{D_i}(\log\Delta_i))\to0.$$
\end{itemize}
\end{lemma}

\begin{proof} We will construct the filtration globally and then verify its properties locally.

On locally free sheaves on $X$, the kernel of the restriction morphism $\iota_{B,*}\iota_B^*\to\iota_{B',*}\iota_{B'}^*$ is given by tensor product with $$\mathcal{J}_0:=\OO_X(-B')/\OO_X(-B).$$ This has the finite filtration $\mathcal{J}_c:=\OO_X(-B'-cD_i)/\OO_X(-B)$, with associated graded pieces $\mathrm{Gr}_c\mathcal{J}\simeq\iota_{i,*}\iota_i^*\OO_X(-B'-cD_i)=\iota_{i,*}F_c$. Tensoring with $T_X(\log D)\otimes L^{\otimes\bullet}$ gives the required filtration of the kernel, namely $\mathcal{K}_c:=\mathcal{J}_c\otimes T_X(\log D)\otimes L^{\otimes\bullet}$. The projection formula identifies the associated graded pieces as
$$\mathrm{Gr}_c\mathcal{K}\simeq\iota_{i,*}\big(F_c\otimes\iota_i^*T_X(\log D)\otimes(\iota_i^*L)^{\otimes\bullet}\big).$$
The action factors through restriction to $\mathcal{S}(D_i)$ because
$$\big[\OO_X(-D_i)T_X(\log D)\otimes L^{\otimes p},\OO_X(-B'-cD_i)T_X(\log D)\otimes L^{\otimes n}\big]$$
$$\subset\OO_X(-B'-(c+1)D_i)T_X(\log D)\otimes L^{\otimes(p+n)}.$$
Indeed logarithmic vector fields preserve both the ideal sheaf factors.

We turn to the second assertion. On $D_i$ we tensor the normal-tangent sequence with $F_c\otimes(\iota_i^*L)^{\otimes\bullet}$ to give an exact sequence,
$$0\to F_c\otimes(\iota_i^*L)^{\otimes\bullet}e\to F_c\otimes\iota_i^*T_X(\log D)\otimes(\iota_i^*L)^{\otimes\bullet}\to$$ $$\to F_c\otimes T_{D_i}(\log\Delta_i)\otimes(\iota_i^*L)^{\otimes\bullet}\to 0.$$ After pushing forward to $X$ this supplies the relevant exact sequence at the level of the underlying sheaves. 

Now the remaining claims can be proven locally, with respect to coordinates as in remark \ref{coord}. A local section $w$ of $F_c\otimes\iota_i^*L^{\otimes n}e$, of weight $n$, lifts to a section on $X$ of the form $$\widetilde{w}:=z^{B'}z_i^{c}w_0(z) z^{-nA}e,\,\ (w_0(z)\,\,\mathrm{regular}),$$ and a weight $m$ section, $u$, of the normal subalgebra lifts to some $$\widetilde{u}=u_0(z) z^{-mA}e,\,\, (u_0(z)\,\,\mathrm{regular}).$$ Their bracket can be expanded via the Leibniz rule to give $$[\widetilde{u},\widetilde{w}]=[u_0(z)z^{-mA}e,z^{B'}z_i^cw_0(z)z^{-nA}e]$$ $$=(c-a_in+a_im)z^{B'}z_i^cu_0(z)w_0(z)z^{-(n+m)A}e$$ $$+z^{B'}z_i^{c+1}\OO z^{-(n+m)A}e,$$ where the terms of order $z_i^{c+1}$ are the result of both $e(u_0(z))$ and $e(w_0(z))$ being divisible by $z_i$, since $u_0(z)$ and $w_0(z)$ are both regular. The terms of order $z^Bz_i^{c+1}$ dissappear in the associated graded and we deduce that the action is given by $$u\cdot w=(c-a_in+a_im)uw.$$ The kernel term is then a submodule isomorphic to $\mathcal{M}_{c,a_i}(F_c)$, exactly as desired. The cokernel term is dispensed of similarly, let $$\widetilde{w}:=z^{B'}z_i^cw_0(z)vz^{-nA},\,\, (w_0(z)\,\,\mathrm{regular},v\in T_{X}(\log D)),$$ be a lift to $X$ of a weight $n$ local section, $w$, of the extension term,  $$F_c\otimes\iota_i^*T_X(\log D)\otimes\iota_i^*L^{\otimes n}.$$ Then we just compute directly that $$[\widetilde{u},\widetilde{w}]=[u_0(z)z^{-mA}e,z^{B'}z_i^cw_0(z)vz^{-nA}]$$ $$=(c-a_in)z^{B}z_i^cu_0(z)w_0(z)vz^{-(n+m)A}$$ $$+u_0(z)w_0(z)[e,v]z^{-(n+m)A}-v(u_0(z))z^{B'}z_i^cw_0(z)ez^{-(n+m)A}.$$ The second term vanishes upon restriction to $D_i$ as $e$ is central in $T_{D_i}(\log\Delta_i)$, the third term vanishes in the quotient by the normal subspace and in the first term we recognize exactly the module $$\mathcal{M}_{c,0}(F_c\otimes T_{D_i}(\log\Delta_i)).$$
\end{proof}

\begin{proof} (\emph{of lemma \ref{q fin gen}}.) We prove by induction on the number of components that $\mathfrak{s}(B)$ is persistently finitely generated whenever $B=\sum_{j\in J}a_jD_j$ for a subset $J\subset I$. The base case, corresponding to the empty divisor, is trivial. Suppose that $B=B'+a_iD_i$ with $D_i$ not a component of $B'$, and assume by induction that $\mathfrak{s}(B')$ is persistently finitely generated. With notation as in lemma \ref{prelim} we set $\mathfrak{K}_c:=H^0(X,\mathcal{K}_c)$ and $\mathfrak{K}:=\mathfrak{K}_0$. 

We may choose $N\gg 0$ so that we have an exact sequence
$$0\to\mathfrak{K}_{\geq N}\to\mathfrak{s}(B)_{\geq N}\to\mathfrak{s}(B')_{\geq N}\to0,$$ a surjection $\mathfrak{s}(B)_{\geq N}\to\mathfrak{s}(D_i)_{\geq N}$ and an identification of associated gradeds $$H^0(X,\mathrm{Gr}_c\mathcal{K})_{\geq N}\simeq \mathrm{Gr}_c\mathfrak{K}_{\geq N}.$$ Indeed this follows from Serre vanishing and ampleness of $A$. The action of $\mathfrak{s}(B)$ on these associated gradeds factors through $\mathfrak{s}(D_i)_{\geq N}=\mathfrak{s}(D_i)$, which is finitely generated by lemma \ref{symb finite}. Considered as modules for the normal subalgebra, they fit into exact sequences
$$0\to\mathfrak{m}_{c,a_i}(F_c)_{\geq N}\to\mathrm{Gr}_c\mathfrak{K}_{\geq N}\to\mathfrak{m}_{c,0}(F_c\otimes T_{D_i}(\log\Delta_i))_{\geq N}\to0.$$
Proposition \ref{norm fin gen} implies that both end terms are finitely generated for $\mathfrak{n}(D_i)_{\geq N}$. Hence so is the middle term, and therefore it is finitely generated as a module for $\mathfrak{s}(D_i)_{\geq N}$. We can now apply lemma \ref{tech}, with notational substitutions as follows; $$\g\leftrightarrow\mathfrak{s}(B)_{\geq N},$$ $$\mathfrak{h}\leftrightarrow\mathfrak{s}(B')_{\geq N},$$ $$\mathfrak{s}\leftrightarrow\mathfrak{s}(D_i)_{\geq N}$$ and filtration giveb by  $\mathfrak{K}_{c,\geq N}$. Finite generation of $\mathfrak{h}$ is the inductive hypothesis, so the lemma gives finite generation of $\mathfrak{s}(B)_{\geq N}$, from which the claim follows easily. \end{proof}

\section{Proving the main theorem}

We are going to prove our main theorem in this section. This will require some fairly standard algebraic geometry, notably we will rely on resolution of singularities. Below when we say that $D$ \emph{supports} a divisor $C$ we mean that the support of $C$ is \emph{all of} $D$. 
\begin{lemma}\label{blowups} Let $U$ be a smooth affine variety over $\mathbf C$. Then there exists a smooth projective compactification $o:U\to X$ such that the boundary $D=X\setminus U$ is a simple normal crossings divisor and supports an effective ample divisor $A=\sum_i a_iD_i.$
\end{lemma}

\begin{proof} Embed $U$ as a closed subvariety of $\mathbf A^N$ and take its projective closure $X_0\subset\mathbf P^N$. The hyperplane at infinity defines an effective ample Cartier divisor $H$ with complement $U$. By Hironaka's theorem, \cite{Hironaka}, there is a sequence of blow-ups supported over $H$ giving a smooth projective compactification $X$ of $U$ with simple normal crossings boundary.

It suffices then to show that the property that the boundary support an effective ample Cartier divisor is preserved under such blow-ups. To this end, let $\pi:V'\to V$ be one such blowup, with exceptional divisor $E$, and suppose that $C$ is an effective ample Cartier divisor whose support is $V\setminus U$. Now, $\OO_{V'}(-E)$ is $\pi$-ample by \cite[Tag 02NS]{Stacks} and then $m\pi^*C-E$ is ample for $m$ sufficiently large by \cite[Tag 0892]{Stacks}. Moreover, the center is contained scheme-theoretically in $rC$ for some $r>0$, so $r\pi^*C-E$ is effective. For $m>r$, the divisor $$m\pi^*C-E=(m-r)\pi^*C+(r\pi^*C-E)$$ is therefore effective with support exactly $V'\setminus U$. 

\end{proof}

\begin{theorem}\label{main} Let $U$ be a smooth affine variety over $\mathbf C$. Then $\mathrm{Vect}(U)=H^0(U,T_U)$ is finitely generated as a Lie algebra.
\end{theorem}

\begin{proof} Choose a compactification as in Lemma \ref{blowups}, and let us write $\g=H^0(U,T_U).$ Then as above we have the filtration $$F^n\g=H^0(X,T_X(\log D)(nA))).$$ Serre vanishing identifies the associated graded, in high degrees, with $\mathfrak{s}(A)$, which we have proven is persistently finitely generated in lemma \ref{q fin gen}. The main claim follows immediately.\end{proof}

Of course, along the way we have proven the following result, which we stress\footnote{Somewhat defensively, in light of the appendix...} is formally quite a bit stronger.

\begin{theorem}\label{real main} Let $X$ be smooth projective and let $D$ an SNC divisor, and $A$ an ample effective divisor whose support is $D$. Consider the filtered Lie algebra $$\mathfrak{g}=H^0(X,T_X(\log D)(\infty A)),\,\, F^n\g:=H^0(X,T_X(\log D)(nA)),$$  Then for any $N\geq 0$ the graded Lie algebra $\mathrm{Gr}_{\geq N}^F\g$ is finitely generated. \end{theorem}

\appendix

\section{Another proof}

What follows is due to ChatGPT 6 (Astra). It was prompted, out of self-doubt and morbid curiosity, by the author \emph{after} the writing of the main text. The proof which follows has been minimally rewritten by the author, with no additions of any mathematical substance. The main idea is very simple, we try to recover the $A$-linear structure on vector fields purely in terms of the adjoint action.   

\begin{remark} Notation is kept purely algebraic here, as there is almost no geometry in sight. It clashes a little with the notation of the main text, but this should not cause any confusion.\end{remark} 

\begin{theorem}\label{main-appendix}Let $A$ be a smooth finite type $k$-algebra, where the characteristic of $k$ is not $2$. Then $\g:=\mathrm{Der}_k(A)$ is finitely generated.\end{theorem}

The proof is based on the following lemma.

\begin{lemma} Let $B\subset\mathrm{End}_k(\g)$ denote the unital subalgebra spanned by the operators $\mathrm{ad}(\delta)$, for $\delta\in\g$. Then $B$ contains $\mathrm{End}_A(\g)\subset\mathrm{End}_k(\g)$. \end{lemma}
\begin{proof}
Define, for $\delta\in\g$ and $f\in \OO$, the operator $T_{\delta,f}:\g\to\g$ by $$T_{\delta,f}(\eta)=\delta(f)\eta(f)\delta.$$ An explicit computation confirms the following identies: $$4T_{\delta,f}=2\mathrm{ad}(f\delta)^2-\mathrm{ad}(f^2\delta)\mathrm{ad}(\delta)-\mathrm{ad}(\delta)\mathrm{ad}(f^2\delta),$$ $$T_{(h+1)\delta,f}-T_{(h-1)\delta,f}=4hT_{\delta,f}.$$ In particular we see that any operator $hT_{\delta,f}$ lies inside $B$, using that $k$ does not have characteristic $2$. It suffices now to check that the submodule generated by the operators $T_{\delta,f}$ is all of $\mathrm{End}_A(\g)$. It suffices to check this on all points $p$. Smoothness implies $$\mathrm{End}_A(\g)\otimes k_p\simeq\mathrm{End}_A(\g_p),$$ where $k_p$ denotes the residue field at $p$ and $\g_p=\g\otimes_Ak_p$ the fiber of the tangent sheaf at $p$. For any $\delta_p$ in $\g_p$ we can lift to some $\delta\in\g$. We can also lift any $\alpha_p\in \g_p\dl$ to some some $df$ in the cotangent sheaf $\mathrm{Hom}_A(\g,A)$. The operator $T_{\delta,f}$ then restricts to $$\alpha_p(v)v\otimes\alpha\in \g_p\otimes\g_p\dl\simeq\mathrm{End}_k(\g_p).$$ These obviously span, so we are done.
\end{proof}
\begin{proof} (\emph{Of Theorem \ref{main-appendix}}). Let $\{a_i\}_i$ be a finite set of generators of $A$. Then each $a_i\in A\subset\mathrm{End}_A(\g)\subset\mathrm{End}_k(\g)$ can be expressed as a polynomial in finitely many operators $\mathrm{ad}(\delta_i^j)$. Let $\nu_l$ denote generators of $\g$ as an $A$-module. Then we claim that $$S:=\{\delta_i^j,\nu_l\}_{i,j,l}$$ generates $\g$. Indeed let $\mathfrak{h}\subset\g$ be the sub Lie algebra generated by $S$. It contains $A$-module generators for $\g$ by definition, so it suffices to show that it is an $A$ submodule. For this it suffices to show that it is stable under multiplication by the generators $a_i$. This follows immediately from the fact that each $a_i$ is expressable as a polynomial in the operators $\mathrm{ad}(\delta_i^j)$, and the fact that $\mathfrak{h}$ is obviously stable under the adjoint action of each $\delta_i^j$. \end{proof}

\begin{remark} That characteristic $\neq 2$ is necessary is easy to see. $\mathrm{Vect}(\aff^1_k)$ is not finitely generated in characteristic $2$. Indeed recall the basis $L_n=-z^{n+1}\partial$ with respect to which we have brackets $$[L_n,L_m]=(n-m)L_{n+m}.$$ Then if $n+m$ is even, $n-m$ is also even and so vanishes in $k$, so we only obtain terms $L_{2l+1}$ as commutators. \end{remark}

\end{document}